\documentclass{amsart}
\usepackage{amsmath,amssymb,amsthm}
\newtheorem{theorem}{Theorem}[section]
\newtheorem{definition}[theorem]{Definition}
\newtheorem{lemma}[theorem]{Lemma}
\newtheorem{proposition}[theorem]{Proposition}
\newtheorem{corollary}[theorem]{Corollary}
\newtheorem{example}[theorem]{Example}
\newtheorem{remark}[theorem]{Remark}
\numberwithin{equation}{section}

\newcommand{\fra}[2]{\displaystyle{\frac{#1}{#2}}}
\newcommand{\ka}{\kappa}
\newcommand{\la}{\lambda}
\newcommand{\w}[1]{\widetilde{#1}}

\begin{document}
\title[Almost cosymplectic and almost Kenmotsu $(\kappa,\mu,\nu)$-spaces]{The curvature tensor of almost cosymplectic and almost Kenmotsu $(\kappa,\mu,\nu)$-spaces}

\author[A. Carriazo]{Alfonso Carriazo}
\address[Alfonso Carriazo and Ver\'{o}nica Mart\'{\i}n-Molina]{Department of Geometry and Topology\\
Faculty of Mathematics\\
University of Sevilla\\
Aptdo. de Correos 1160\\
41080 -- Sevilla, SPAIN}
\thanks{Both authors are partially supported by the MTM2011-22621 grant from the MEC (Spain) and by the PAI group FQM-327 (Junta de Andaluc\'{\i}a, Spain).}
\email[Alfonso Carriazo]{carriazo@us.es}
\author[V. Mart\'{\i}n-Molina]{Ver\'{o}nica Mart\'{\i}n-Molina}
\email[Ver\'{o}nica Mart\'{\i}n-Molina]{veronicamartin@us.es}

\subjclass[2010]{53C15, 53C25}
\keywords{generalized $(\kappa, \mu)$-space, $(\kappa,\mu, \nu)$-space, generalized Sasakian space form, almost cosymplectic, almost Kenmotsu}

\begin{abstract}
\noindent We study the Riemann curvature tensor of  \emph{$(\kappa,\mu,\nu)$-spaces} when they have almost cosymplectic and almost Kenmotsu structures, giving its writing explicitly. This leads to the definition and study of  a natural generalisation of the contact metric \emph{$(\kappa,\mu,\nu)$-spaces}. We present examples or obstruction results of these spaces in all possible cases.
\end{abstract}

\maketitle

\section{Introduction}

The study of the curvature tensor of a Riemannian manifold as a tool to classify it constitutes an important part of Differential Geometry.  In particular, many advances have been made recently when the manifold is a generalized $(\ka,\mu)$-space  or a $(\ka,\mu,\nu)$-space with constant $\phi$-sectional curvature, called generalized $(\ka,\mu)$-space form and $(\ka,\mu,\nu)$-space form, respectively. These manifolds are considered of great importance by all researchers who are currently working on contact metric geometry and related topics. The contact metric $(\ka,\mu)$-spaces were originally introduced (under a different name) by D. E. Blair, T. Koufogiorgos and V. J. Papantoniou in \cite{bisr} as those contact metric manifolds satisfying the equation
\begin{equation}\label{kappamu}
R(X,Y)\xi = \kappa \{ \eta (Y)X-\eta (X) Y\} + \mu\{\eta (Y)hX-\eta (X)hY\},
\end{equation}
for every  vector fields $X,Y$ on $M$, where $\kappa$ and $\mu$ are constants,
$h=1/2 L_\xi \phi$ and $L$ is the usual Lie derivative. These
spaces include the Sasakian manifolds ($\kappa=1$ and $h=0$), but
the non-Sasakian examples have proven to be even more interesting.
In \cite{kouf2000}, the contact metric \emph{generalized $(\ka,\mu)$-spaces} were introduced as contact metric manifolds satisfying the equation \eqref{kappamu} with $\ka,\mu$ functions.

The curvature tensor of a contact metric $(\kappa,\mu)$-space form was shown by T. Koufogiorgos in \cite{kouf97} to have the form
\begin{equation*}
R=\fra{F+3}{4} R_1+\fra{F-1}{4} R_2+ \left( \fra{F+3}{4}- \ka \right) R_3+R_4+\fra{1}{2} R_5+(1-\mu)R_6,
\end{equation*}
where $R_1, \ldots, R_6$ are the tensors
\begin{equation}\label{R1-R6}
\begin{aligned}
R_1(X,Y)Z &=g(Y,Z)X-g(X,Z)Y, \\
R_2(X,Y)Z &=g(X,\phi Z)\phi Y-g(Y,\phi Z)\phi X+2g(X,\phi Y)\phi Z, \\
R_3(X,Y)Z &= \eta(X)\eta(Z)Y -\eta(Y)\eta(Z)X +g(X,Z)\eta(Y)\xi-g(Y,Z)\eta(X)\xi. \\
R_4(X,Y)Z &=g(Y,Z)hX-g(X,Z)hY+g(hY,Z)X-g(hX,Z)Y,  \\
R_5(X,Y)Z &=g(hY,Z)hX-g(hX,Z)hY+g(\phi hX,Z) \phi hY-g(\phi h Y,Z) \phi hX,\\
R_6(X,Y)Z &= \eta(X)\eta(Z)hY-\eta(Y)\eta(Z)hX+g(hX,Z) \eta (Y) \xi - g(hY,Z) \eta (X) \xi,
\end{aligned}
\end{equation}
for every vector fields $X,Y,Z$ on $M$, where $F$ is the constant $\phi$-sectional curvature.

This result led the authors (jointly with M. M. Tripathi) to define in \cite{nuestro} a \emph{generalized $(\ka,\mu)$-space form} as an almost contact metric manifold whose curvature tensor can be written as
\begin{equation} \label{def-space}
R=f_1R_1+f_2R_2+f_3R_3+f_4R_4+f_5R_5+f_6R_6,
\end{equation}
where $f_1, \ldots, f_6$ are functions on $M$ and $R_1,\ldots,R_6$ the previously defined tensors. They denoted them by $M(f_1, \ldots,f_6)$ and shortened their name to \emph{g.$(\ka,\mu)$-s.f.} In that same paper, they also studied these spaces with contact metric structure in every dimension, giving examples of them for all cases.

In a later work, \cite{nuestro2}, they continued studying them when they have contact metric structure and showing what happens when they are $D_a$-homothetically deformed, which preserves their structure in dimension $3$, although with different functions $f_1,\ldots,f_6$. In dimension greater than or equal to $5$, a small change in the definition was needed, which meant the introduction of \emph{generalized $(\ka,\mu)$-spaces with divided $R_5$}, denoted by $M(f_1, \ldots,f_{5,1},f_{5,2},f_6)$, as almost contact metric manifolds with curvature tensor of the form
\begin{equation*}
R=f_1R_1+f_2R_2+f_3R_3+f_4R_4+f_{5,1}R_{5,1}+f_{5,2} R_{5,2}+f_6R_6,
\end{equation*}
with $f_1,\ldots, f_6$ functions on $M$ and $R_{5,1}, R_{5,2}$ the tensors
\begin{align}
R_{5,1}(X,Y)Z&=g(hY,Z)hX-g(hX,Z)hY,     \label{def-R51}\\
R_{5,2}(X,Y)Z&=g(\phi h Y,Z) \phi hX- g(\phi hX,Z) \phi hY.     \label{def-R52}
\end{align}
It is obvious that $R_5=R_{5,1}-R_{5,2}$, so the \emph{generalized $(\ka,\mu)$-spaces with divided $R_5$} include the \emph{generalized $(\ka,\mu)$-spaces}. If we apply a $D_a$-homothetic deformation to a contact metric \emph{generalized $(\ka,\mu)$-space with divided $R_5$} of any dimension, we obtain another one (with different functions), which provides us with infinitely many examples.

Going beyond \emph{generalized $(\kappa,\mu)$-spaces}, T. Koufogiorgos, M. Markellos and V. J. Papantoniou introduced in \cite{kouf2008} the notion of \emph{$(\kappa,\mu,\nu)$-contact metric manifold}, where now the equation to be satisfied is
\begin{equation} \label{kappamunu}
\begin{split}
R( X,Y) \xi &=\kappa \{ \eta ( Y) X-\eta ( X) Y\} +\mu \{ \eta ( Y) hX-\eta ( X) hY\} \\
&+\nu \{ \eta ( Y) \phi hX-\eta ( X) \phi hY\},
\end{split}
\end{equation}
for some smooth functions $\kappa, \mu $, and $\nu $ on $M$.

They proved that this type of manifold is  intrinsically related to the harmonicity of the Reeb vector on contact metric $3$-manifolds. In dimension greater than or equal to $5$, these manifolds  must be $(\ka,\mu)$-spaces but in dimension $3$ there are examples with non-zero $\nu$ and non-constant  $\ka$ or $\mu$. Some other authors have also studied manifolds satisfying condition \eqref{kappamunu}, but with a non-contact metric structure.
Such is the case of  P. Dacko and Z. Olszak, who in \cite{dacko2005}  defined an  \emph{almost cosymplectic $(\kappa,\mu,\nu)$-space} as an almost cosymplectic manifold that satisfies \eqref{kappamunu}, but with $\kappa,\mu,\nu$ functions varying exclusively in the direction of $\xi$. Later, they gave in \cite{dacko2005-b} examples of this type of manifolds.
G. Dileo and A. M. Pastore  analysed in \cite{dileo09} the  $(\kappa,\mu)$-spaces and the  $(\kappa,0,-\mu)$-spaces with almost Kenmotsu structure. Lastly, H. \"{O}zt\"{u}rk, N. Aktan and C. Murathan studied in \cite{murathan-arxiv} the almost $\alpha$-cosymplectic $(\kappa,\mu,\nu)$-spaces under different conditions (like $\eta$-parallelism) and gave an interesting example in dimension $3$.  Some of their results are used in this paper but we later employ a different approach to the study of $(\kappa,\mu,\nu)$-spaces, concentrating on the writing of the curvature tensor and its relation to generalized $(\ka,\mu,\nu)$-space forms.

Therefore, from now on  we will denote by $(\ka,\mu)$-space or $(\ka,\mu,\nu)$-space an almost contact metric manifold satisfying equations \eqref{kappamu} or \eqref{kappamunu}, respectively,   and we will  specify  its structure explicitly.

Recently, the authors (jointly with K. Arslan and C. Murathan) proved in  \cite{nuestro-murathan}  that the curvature tensor of a  $(\ka,\mu,\nu)$-contact metric manifold of dimension $3$ is not unique  and can be written, among others, as
\begin{equation*}
R=FR_1+(F-\ka)R_3+\mu R_4+\nu R_7=FR_1+(F-\ka)R_3+\mu R_4-\nu R_8,
\end{equation*}
where $R_7$ and $R_8$ are the tensors
\begin{align}
R_7&= g(Y,Z)\phi h X-g(X,Z)\phi h Y+g(\phi h Y,Z)X-g(\phi h X,Z),Y,     \label{def-R7}\\
R_8&=\eta(X)\eta(Z)\phi h Y- \eta(Y)\eta(Z)\phi h X+g(\phi h X,Z\eta(Y)\xi-g(\phi h Y,Z)\eta(X)\xi.     \label{def-R8}
\end{align}

It is important to note that $R_7 = -R_8$ in dimension 3 but not in general. This led to the introduction in the same paper of the \emph{generalized $(\ka,\mu,\nu)$-space forms} as those almost contact metric manifolds whose curvature tensor can be written as
\begin{equation} \label{def-space-nu}
R=f_1R_1+f_2R_2+f_3R_3+f_4R_4+f_5R_5+f_6R_6+f_7R_7+f_8R_8,
\end{equation}
where $R_1, \ldots,R_8$ are the tensors previously seen in \eqref{R1-R6}, \eqref{def-R7} and \eqref{def-R8}. They shortened their  name to \emph{g.$(\ka,\mu,\nu)$-s.f.} and denoted them by $M(f_1,\ldots,f_8)$. They also studied these spaces with contact metric structure, giving examples or proving their non-existence  in every dimension.

Despite their technical appearance, there are good reasons for studying  contact metric $(\kappa,\mu,\nu)$-spaces and, therefore, \emph{generalized $(\kappa,\mu,\nu)$-space forms}.
The first is that the condition \eqref{kappamunu} remains invariant under D-homothetic deformations, although the values $\kappa$, $\mu$ and $\nu$ may change. Moreover, these manifolds provide non-trivial examples of some remarkable classes of contact Riemannian manifolds, like CR-integrable contact metric manifolds, H-contact manifolds  and harmonic contact
metric manifolds. It is worth noting that there are non-trivial examples of such Riemannian
manifolds, the most important being the tangent sphere bundle of any Riemannian
manifold of constant sectional curvature with its standard contact metric structure. Finally, in some cases, the formula \eqref{kappamunu} determines the curvature tensor field completely, which will be written in terms of some of the tensors $R_1,\ldots,R_8$.

%%%
This paper is organised in two additional sections. In the first one  we present some background which is necessary in order to follow this work. In the second one we study the  $(\ka,\mu,\nu)$-spaces with almost cosymplectic and almost Kenmotsu structures, giving explicitly the writing of their curvature tensors. This will lead to the definition of \emph{g.$(\ka,\mu,\nu)$-s.f. with divided $R_5$}, of which we will provide examples or obstruction results in all possible cases.

\section{Preliminaries}\label{section-preliminaries}

In this section, we recall some general definitions and basic
formulas which will be used later. For more background on almost
contact metric manifolds, we recommend the reference
\cite{blairb}.

An odd-dimensional Riemann manifold $(M,g)$ is said to be an
\emph{ almost contact metric manifold} if there exist on $M$ a $(1,1)$-tensor field $\phi $, a vector field $\xi $ (called the \emph{structure vector
field}) and a 1-form $\eta $ such that $\eta (\xi )=1$, $\phi
^{2}X=-X+\eta ( X) \xi $ and $g(\phi X,\phi Y)=g(
X,Y) -\eta ( X) \eta ( Y) $ for any
vector fields $X,Y$ on $M$. In particular, in an almost contact
metric manifold we also have $\phi \xi =0$ and $\eta \circ \phi
=0$.

Such a manifold is said to be a \emph{ contact metric manifold} if ${\rm d}%
\eta =\Phi $, where $\Phi ( X,Y) =g(X,\phi Y)$ is the \emph{fundamental $2$-form} of $M$.

On the other hand, the almost contact metric structure of $M$ is
said to be \emph{ normal} if the Nijenhuis torsion $[\phi ,\phi ]$\
of $\phi $ equals $-2{\rm d}\eta \otimes \xi $. A normal contact
metric manifold is called a \emph{ Sasakian manifold}. It can be
proved that an almost contact metric manifold is Sasakian if and
only if
\begin{equation}
(\nabla _{X}\phi )Y=g( X,Y) \xi -\eta ( Y) X
\label{eq-Sas}
\end{equation}%
for any vector fields $X,Y$ on $M$. Moreover, for a Sasakian
manifold the following equation holds:
\[ R( X,Y) \xi =\eta ( Y) X-\eta ( X) Y.\]

Given an almost contact metric manifold $(M,\phi ,\xi ,\eta ,g)$, a $%
\phi $-\emph{section} of $M$ at $p\in M$ is a section $\Pi
\subseteq
T_{p}M $ spanned by a unit vector $X_{p}$ orthogonal to $\xi _{p}$, and $%
\phi_p X_{p}$. The $\phi $-\emph{sectional curvature of} $\Pi $ is
defined by $K(X,\phi X)=R(X,\phi X,\phi X,X)$. A Sasakian manifold
with constant $\phi $-sectional curvature $c$ is called a \emph{
Sasakian space form}. In such a case, its Riemann curvature
tensor is given by equation $R=f_1R_1+f_2R_2+f_3R_3$ with
functions $f_{1} = (c+3)/4$, $f_{2} = f_{3}= (c-1)/4$ and $R_1, R_2$ and $R_3$ the tensors defined in \eqref{R1-R6}.

It is well known that on a contact metric manifold $( M,\phi
,\xi ,\eta ,g) $, the tensor $h$, defined by $2h=L_\xi
\phi$, satisfies the following relations \cite{bisr}:
\begin{equation}
h\xi =0,\quad \nabla _{X}\xi =-\phi X-\phi hX,\quad h\phi =-\phi
h, \quad {\rm tr}h=0,\quad \eta \circ h=0. \label{eq-cont-h}
\end{equation}%
Therefore, it follows that a contact metric manifold is $K$-contact if
and only if $h=0$.

An almost contact metric manifold is said to be \emph{almost cosymplectic}  if $d \eta=0$ and $d \Phi=0$. A normal almost cosymplectic manifold is cosymplectic.  We will say that an almost contact metric manifold is \emph{almost Kenmotsu} if $d \eta=0$ and $d \Phi=2 \eta \wedge \Phi$. A normal almost Kenmotsu manifold is Kenmotsu. In \cite{kim}, T. W. Kim and H. K. Pak defined the notion of \emph{almost $\alpha$-cosymplectic} manifold as such an almost contact metric manifold satisfying $d \eta =0$ and $d \Phi=2 \alpha \eta \wedge \Phi$. These manifolds include trivially the almost cosymplectic   ($\alpha=0$) and almost Kenmotsu ones ($\alpha=1$). A normal almost $\alpha$-cosymplectic manifold is $\alpha$-cosymplectic.

Similar formulas to the ones we had in the contact metric case also hold on $\alpha$-cosymplectic manifolds, where we know that $h$ is a symmetric operator satisfying that \cite{kim}:
\begin{equation}\label{nablaxi}
h  \xi=0, \quad \nabla _X \xi = -\alpha \phi X- \phi h X, \quad h\phi =-\phi h,  \quad tr h=0.
\end{equation}

These results have also been proved when $\alpha=0$ (almost cosymplectic manifolds) in \cite{endo94} and when $\alpha=1$ (almost Kenmotsu manifolds) in \cite{dileo07}.

\section{Almost cosymplectic or almost Kenmotsu generalized $(\ka,\mu,\nu)$-space forms}

In this section  we will study how $(\ka,\mu,\nu)$-spaces behave when they have almost cosymplectic or almost Kenmotsu structures. Firstly, we will present some results from \cite{murathan-arxiv} which are true for $\alpha$-cosymplectic manifolds and therefore on almost cosymplectic (when $\alpha=0$) or almost Kenmotsu ones (when $\alpha=1$). Later, we will study both structures separately as we will use different approaches in order to obtain the writing of the curvature tensor of a $(\kappa,\mu,\nu)$-space.

\begin{proposition}[\cite{murathan-arxiv}]\label{prop-basica}
Given $M$ an almost $\alpha$-cosymplectic $(\ka,\mu,\nu)$-space, then
\begin{equation}\label{h-phi}
h^2=(\ka+\alpha^2) \phi^2,
\end{equation}
hence  $\ka \leq -\alpha^2 $ and  $\ka=-\alpha^2$  if and only if $h=0$. Moreover, the next formulas are satisfied
\begin{equation}
\xi(\ka)=2 ( \ka+\alpha^2)(\nu-2\alpha),  \label{xi-ka}
\end{equation}
\begin{equation}\label{eq-R-xi-x-y}
\begin{aligned}
R(\xi,X)Y&=\ka (g(X,Y)\xi-\eta(Y)X) + \mu (g (h X,Y)\xi-\eta(Y)h X)\\
&+\nu (g(\phi h X,Y)\xi-\eta(Y)\phi h X),
\end{aligned}
\end{equation}
\begin{equation}\label{nabla-phi-h}
\begin{aligned}
(\nabla_Y \phi h)X-(\nabla_X \phi h)Y&=(\ka+\alpha^2) (\eta(Y)X-\eta(X)Y)\\
&+\mu(\eta(Y)h X-\eta(X)h Y) +(\nu-\alpha)(\eta(Y)\phi h X-\eta(X)\phi h Y),
\end{aligned}
\end{equation}
\begin{align}
(\nabla_X \phi)Y&=g(\alpha \phi X+h X,Y)\xi-\eta(Y)(\alpha \phi X+ h X), \label{nabla-phi}\\
(\nabla_Y  h)X-(\nabla_X  h)Y&=(\ka+\alpha^2) (\eta(X)\phi Y-\eta(Y)\phi X+2g( X,\phi Y)\xi), \label{nabla-h}\\
&+\mu(\eta(X)\phi h Y-\eta(Y)\phi h X)  +(\alpha-\nu)(\eta(X) h Y-\eta(Y) h X), \nonumber
\end{align}
for any $X,Y$ vector fields on $M$.
\end{proposition}

\begin{theorem}[\cite{murathan-arxiv}]\label{teo-xi}
On an almost $\alpha$-cosymplectic $(\ka,\mu,\nu)$-space of dimension greater than or equal to $5$, the functions $\kappa$,  $\mu$ and $\nu$ only vary in the direction of $\xi$, i.e. $X(\ka)=X(\mu)=X(\nu)=0$ for every vector field $X$ orthogonal to $\xi$.
\end{theorem}

%%%%%%%%%%%%%%%%%%%%%%%%%%%%%%%%%%%%%%%%%%%%%%%%%%%%%%%%%%%%%%%%%%%%%%%%%%%%%%%%%%
%%Estudiaremos el caso casi-cosimpléctico
%%%%%%%%%%%%%%%%%%%%%%%%%%%%%%%%%%%%%%%%%%%%%%%%%%%%%%%%%%%%%%%%%%%%%%%%%%%%%%%%%%

We will now focus  on the almost cosymplectic  $(\ka,\mu,\nu)$-spaces, which have already been  studied by other authors in some particular cases. For $\mu=\nu=0$, P. Dacko published   \cite{dacko2000}, where he proved that $\ka$ must be constant, and H. Endo presented multiple results in   \cite{endo94} and \cite{endo96}. This last author also examined in  \cite{endo97} and \cite{endo2002} the $(\ka,\mu,\nu)$-spaces with constant $\ka$,  $\mu$ and  $\nu=0$. Later, P. Dacko and Z. Olszak studied in \cite{dacko2005} and \cite{dacko2005-b} the almost cosymplectic $(\ka,\mu,\nu)$-spaces with $\ka,\mu$ and $\nu$ functions that only vary on the direction of the vector  field $\xi$.

We now present a result that is valid for any functions $\ka,\mu$ and $\nu$.

\begin{proposition}\label{prop-basica-cos}
Let $M^{2n+1}$ be an almost cosymplectic $(\ka,\mu,\nu)$-space. Then
\begin{align}
\nabla_\xi \phi h=&\mu h+\nu \phi h, \label{nabla-xi-phi-h-cos}
\end{align} for every $X,Y$ differentiable vector fields on $M$.

If $\ka=0$, then $h=0$ and $M$ is a cosymplectic manifold if its dimension is $3$.

If $\ka <0$, the eigenvalues of $h$ are $0$ (with multiplicity $1$) and $\pm \la=\pm \sqrt{-\ka}$ (each one with multiplicity $n$). Moreover,  $\mu=-2g(\nabla_\xi X,\phi X)$ holds for every $X$ eigenvector of $h$ associated to the eigenvalue $\la$ or $-\la$.
\end{proposition}
\begin{proof}
If $\ka=0$, then $h=0$ and, by virtue of \eqref{nablaxi}, we obtain that $\nabla \xi=0$. When $M$ is of dimension $3$, it is enough to apply Corollary 5.6. of \cite{olszak81}, which says that an almost contact metric manifold of dimension $3$ is cosymplectic if and only if $\nabla \xi=0$.

%%dem de la ecuación:
Choosing $Y=\xi$ in \eqref{nabla-phi-h} (with $\alpha=0$), we deduce that
\[(\nabla_\xi \phi h)X-(\nabla_X \phi h)\xi=-\ka \phi^2 X+\mu h X+\nu \phi h X.\]

On the other hand, using \eqref{nablaxi} and \eqref{h-phi}, then
\[(\nabla_X \phi h)\xi=\nabla_X \phi h \xi-\phi h \nabla_X \xi=\phi h \phi h X=-\phi^2 h^2 X=h^2 X=\ka \phi^2 X,\]
and substituting in the last equation we obtain \eqref{nabla-xi-phi-h-cos}.

%%dos ultimas igualdades:
If  $\ka<0$, by the properties of $\phi$ and of $h$ we know that the eigenvalues of $h$ are $0$ (with multiplicity $1$) and $\pm \la=\pm \sqrt{-\ka} \neq 0$ (each one with multiplicity $n$). Let us take a unit vector field $X$, eigenvector of $h$ associated to the eigenvalue $\la=\sqrt{-\ka} $ (denoted by $X \in D(\la)$), which is orthogonal to $\xi$ and satisfies by \eqref{nabla-xi-phi-h-cos}  that
\begin{equation} \label{unacualquiera}
(\nabla_\xi \phi h)X=\mu hX+\nu \phi h X=\la \mu X+\la \nu \phi X.
\end{equation}

Taking the inner product of \eqref{unacualquiera} with $\phi X$  gives 
\begin{equation*}
\begin{aligned}
\la \nu %&=g((\nabla_\xi \phi h)X, \phi X)=g(\xi(\la) \phi X +\la \nabla_\xi \phi X-\phi h \nabla_\xi X,\phi X)\\
&=\xi(\la)+\la g( \nabla_\xi \phi X,\phi X)-g(\phi h  \nabla_\xi X,\phi  X)=\xi(\la)=-\fra{1}{2 \la} \xi(\ka),
\end{aligned}
\end{equation*} from where \eqref{xi-ka} follows with $\alpha=0$.

Taking now the product of \eqref{unacualquiera} with $X$ gives 
\begin{equation*}
\begin{aligned}
\la \mu %&=g((\nabla_\xi \phi h)X,X)=g(\xi(\la) \phi X +\la \nabla_\xi \phi X-\phi h \nabla_\xi X,X)\\
&=\la g( \nabla_\xi \phi X,X)-g( \nabla_\xi X,\phi h X)=-2 \la  g(\nabla_\xi X,\phi X).
\end{aligned}
\end{equation*}
Moreover, we know by hypothesis that $\la \neq 0$, so $\mu=-2g(\nabla_\xi X,\phi X)$ for every unit $X \in D(\la)$.

If we take $X$ a unit eigenvector of $h$ associated to the eigenvalue $-\la$, we get again the same two equations.
\end{proof}

%\begin{remark}\label{notakapa2-cos}
%When we write that $M$ is an almost cosymplectic $(\ka,\mu,\nu)$-space with $\ka=0$ or $\ka<0$, we mean that the function $\ka$ satisfies that condition for every point $M$. This is not a problem because it would be enough to take $B=\{ p \in M \ / \ \ka(p)=0 \}$ and redefine $M$ as $M-B$, which would be an open subset of $M$ and therefore also a manifold.
%\end{remark}

%%%%%%%%%%%%%%%%%%%%%%%%%%%%%
%%Las funciones sólo varían en la dirección de \xi si la dimensión es mayor o igual que 5
%%%%%%%%%%%%%%%%%%%%%%%%%%%
By Theorem \ref{teo-xi}, if an almost cosymplectic $(\ka,\mu,\nu)$-space is of dimension greater than or equal to  $5$, the functions $\ka$,$\mu$ and $\nu$ only vary in the direction of $\xi$, hence we can use the results of \cite{dacko2005} and \cite{dacko2005-b}, some of which  will be summarised below. The case  of dimension $3$ will be studied apart later.
%%%%%%%

\begin{proposition}[\cite{dacko2005}]\label{prop-casi-c}
Let  $M^{2n+1}$ be an almost cosymplectic $(\ka,\mu,\nu)$-space, where $\ka,\mu,\nu$ only vary in the direction of  $\xi$.

If  $\ka=0$ on some point of $M$, then $\ka$ is the identically zero function on  $M$ and $h=0$, so $R(X,Y)\xi=0$ for every  $X,Y$ on $M$. Moreover, $M$ is locally the product of an open interval and an almost K\"{a}hler manifold.

If $\ka <0$, then the eigenvalues of $h$ are $0$ (with multiplicity $1$) and $\pm \sqrt{-\ka}$ (each one with multiplicity $n$). %%Además, las hojas de la foliación canónica $\cal{F}$ de $M$ son variedades Kaehlerianas localmente llanas.
In particular, $\ka$ is constant if and only if $\nu=0$.
\end{proposition}

We will  consider a $D$-homothetic deformation of the almost contact metric structure $(\phi,\xi,\eta, g)$ defined as \cite{dacko2005}
\begin{equation} \label{transf-cos}
\overline{\phi}=\phi, \ \overline{\xi}=\fra{1}{\beta}\xi, \ \overline{\eta}=\beta \eta, \  \overline{g}=\alpha g+(\beta^2-\alpha) \eta \otimes \eta,
\end{equation} where $\alpha$ is a positive constant and  $\beta$ a function that only varies in the direction of $\xi$ and is not zero on any point of the manifold.

\begin{proposition}[\cite{dacko2005}]\label{prop-dacko2}
If $(M,\phi,\xi,\eta,g)$ is an almost cosymplectic manifold, the tensor $\overline{h}$ and the Levi-Civita connection $\overline{\nabla}$ of the deformed manifold are related to the original ones the following way:
\begin{gather}
\overline{h}=\fra{1}{\beta} h, \label{h-deformed}\\
\overline{\nabla}_X Y=\nabla_X Y-\fra{\beta^2-\alpha}{\beta^2} g(\phi h X,Y)\xi+ \fra{\xi(\beta)}{\beta} \eta(X) \eta(Y) \xi,  \label{nabla-deformed-cos}
\end{gather} for all $X,Y$ vector fields on $M$.

Hence the almost cosymplectic $(\ka,\mu,\nu)$-spaces with $\ka,\mu,\nu$  varying only in the direction of  $\xi$ are deformed in almost cosymplectic $(\overline{\ka},\overline{\mu},\overline{\nu})$-spaces with
\[\overline{\ka}=\fra{\ka}{\beta^2}, \ \overline{\mu}=\fra{\mu}{\beta}, \ \overline{\nu}=\fra{\nu \beta-\xi(\beta)}{\beta^2},\] where $\overline{\ka},\overline{\mu},\overline{\nu}$ only vary in the direction of $\overline{\xi}$.

Therefore, every almost cosymplectic $(\ka,\mu,\nu)$-spaces with $\ka<0$ can be deformed in an almost cosymplectic  $(-1,\overline{\mu},0)$-space with $\overline{\mu}=\mu / {\sqrt{-\ka}}$.
\end{proposition}

We will now study the curvature tensor of an almost cosymplectic $(\ka,\mu,\nu)$-space. If $\ka=0$, then we know  its local structure by virtue of Proposition \ref{prop-casi-c}.  If $\ka <0$, then it follows from Proposition  \ref{prop-dacko2} that we can  obtain the writing of its curvature tensor by studying the form of an almost cosymplectic $(\ka,\mu)$-space. Using formula \eqref{xi-ka}, we know that these latter spaces satisfy $\xi(\ka)=0$, so in dimensions greater than or equal to $5$, $\ka$ would be constant and $\mu$ would only vary in the direction of $\xi$. This implies that we can also use \cite{endo2002} because, although in that article Endo focuses in almost cosymplectic $(\ka,\mu)$-spaces with  $\ka,\mu \in \mathbb{R}$, a review of the proofs reveals that they are also true if  $\mu$ is not constant but only varies in the direction of $\xi$.  Hence Theorem 3.1 of \cite{endo2002} would look like this in our case:

\begin{theorem}\label{teo-endo}
If $M$ is an almost cosymplectic $(\ka,\mu)$-space with $\ka<0$, where $\ka,\mu$ only vary in the direction of $\xi$, then
\begin{equation*}
\begin{aligned}
&R(X_{\la},Y_\la)Z_{-\la}&=& \quad \ka\{g(\phi Y_\la,Z_{-\la})\phi X_\la -g(\phi X_{\la}, Z_{-\la}) \phi Y_\la\}, \\
&R(X_{-\la},Y_{-\la})Z_{\la}&=& \quad \ka \{g(\phi Y_{-\la},Z_{\la})\phi X_{-\la} -g(\phi X_{-\la}, Z_{\la}) \phi Y_{-\la}\},\\
&R(X_{\la},Y_{-\la})Z_{-\la}&=&\quad -\ka g(X_{\la},\phi Z_{-\la})\phi Y_{-\la},\\
&R(X_{\la},Y_{-\la})Z_{\la}&=& \quad -\ka g(Z_{\la}, \phi Y_{-\la})\phi X_{\la},\\
&R(X_\la,Y_\la)Z_\la&=&  \quad 0,\\
&R(X_{-\la},Y_{-\la})Z_{-\la}&=& \quad 0,
\end{aligned}
\end{equation*}
where $X_{\pm \la},Y_{ \pm \la},Z_{\pm \la}$ are eigenvectors of  $h$ associated to the eigenvalues $\pm \la=\pm \sqrt{-\ka}$.
\end{theorem}

Using the previous theorem and formula \eqref{eq-R-xi-x-y}, we will give explicitly the form of the curvature tensor of an almost cosymplectic $(\ka,\mu)$-space with $\ka<0$.

\begin{theorem}\label{teo-est-c}
Let $M$ be an almost cosymplectic  $(\ka,\mu)$-space of dimension greater than or equal to  $5$ with $\ka < 0$. Then its Riemann curvature tensor can be written as
\[R=-\ka R_3-R_{5,2}-\mu R_6,\]
where $R_3,R_6$ are the tensors defined in \eqref{R1-R6} and $R_{5,2}$ is the one in \eqref{def-R52}.

Therefore, $M$ is a \emph{g.$(\ka,\mu)$-s.f. with divided $R_5$} $M(f_1,\ldots,f_{5,1},f_{5,2},f_6)$ with functions
\[f_1=f_2=0, \ f_3=-\ka, \ f_4=f_{5,1}=0,  \ f_{5,2}=-1  \mbox{ and } f_6=-\mu, .\]
\end{theorem}
\begin{proof}
As $\ka < 0$, we know by Proposition \ref{prop-casi-c} that $TM=D(\lambda) \oplus D(-\lambda) \ \oplus <\xi> $, where $\lambda=\sqrt{-\ka}>0$ . Given a vector field  $X$ on $M$, we can write $X=X_\lambda+X_{-\lambda}+\eta(X)\xi$, where $X_{\pm \lambda}$ is an eigenvector of $h$ associated to the eigenvalue $\pm \lambda$. Then, by the properties of $R$ we obtain that
\begin{align*}
R(X,Y)Z=&R(X_\lambda+X_{-\lambda}, Y_\lambda+Y_{-\lambda})(Z_\lambda+Z_{-\lambda})+\eta(X)R(\xi,Y)Z\\
&+\eta(Y)R(X,\xi)Z+\eta(Z) R(X_\la+X_{-\la},Y_\la+Y_{-\la}) \xi,
\end{align*}
from which, using \eqref{kappamu}, we get
\begin{equation}\label{descomposicion-c}
\begin{split}
R(X,Y)Z=&R(X_\lambda+X_{-\lambda}, Y_\lambda+Y_{-\lambda})(Z_\lambda+Z_{-\lambda})\\
&+\eta(X)R(\xi,Y)Z+\eta(Y)R(X,\xi)Z.
\end{split}
\end{equation}

 It follows from equation \eqref{eq-R-xi-x-y} (with $\nu=0$) and the definition of the tensors $R_1, \ldots,R_6$ that
\begin{equation} \label{enesima-intermedia}
\eta(X)R(\xi,Y)Z+\eta(Y)R(X,\xi)Z=-\ka R_3 (X,Y)Z-\mu R_6(X,Y)Z.
\end{equation}

By Theorem \ref{teo-endo}, we obtain that
\begin{align*}
R( X_\lambda+ X_{-\lambda}, Y_\lambda+Y_{-\lambda})(Z_\lambda+Z_{-\lambda})&=
\ka \{ (g(X_\la, \phi Z_{-\la})  -g(X_{-\la}, \phi Z_{\la}) (\phi Y_\la - \phi Y_{-\la}) \\
& \ -(g(Y_{\la}, \phi Z_{-\la}) - g(Y_{-\la},\phi Z_{\la}))(\phi X_{\la}-\phi X_{-\la}) \}.
\end{align*}
From the decomposition $X=X_\lambda+X_{-\lambda}+\eta(X)\xi$, it can be deduced that $X_\la=\fra{1}{2} \left( X-\eta(X)\xi+\fra{1}{\la}hX \right)$ and that $X_{-\la}=\fra{1}{2} \left( X-\eta(X)\xi-\fra{1}{\la}hX \right)$, hence
\begin{equation} \label{enesima-intermedia2}
\begin{aligned}
R(X_\lambda+&X_{-\lambda}, Y_\lambda+Y_{-\lambda})(Z_\lambda+Z_{-\lambda})=\\
&=\fra{\ka}{\la^2} ( -g(\phi h X,Z)\phi h Y+g(\phi h Y,Z)\phi h X  )=-R_{5,2}(X,Y)Z.
\end{aligned}
\end{equation}

Substituting \eqref{enesima-intermedia} and \eqref{enesima-intermedia2} in \eqref{descomposicion-c},  we conclude that
\[R(X,Y)Z=-\ka R_3(X,Y)Z-\mu R_6(X,Y)Z-R_{5,2}(X,Y)Z,\] for all $X,Y,Z$ vector fields on $M$.
\end{proof}

\begin{remark}
By the previous theorem, every almost cosymplectic $(\ka,0)$-space with constant $\ka<0$ has curvature tensor $R=-\ka R_3-R_{5,2},$ which coincides with Lemma 5 from \cite{dacko2000}.
\end{remark}
\begin{example}\label{eje-cos}
P. Dacko and Z. Olszak gave in \cite{dacko2005-b} models of examples of almost cosymplectic $(-1,\mu,0)$-spaces, which they denoted by $N(\mu)$. By virtue of Theorem \ref{teo-est-c}, these spaces are \emph{g.$(\ka,\mu)$-s.f.'s with divided $R_5$} $M(f_1,\ldots,f_{5,1},f_{5,2},f_6)$ with functions
\[f_1=f_2=0, \ f_3=1, \ f_4=f_{5,1}=0,  \ f_{5,2}=-1  \mbox{ and } f_6=-\mu.\]
\end{example}

We will use now the $D$-homothetic deformations given by  \eqref{transf-cos} in order to obtain from the previous theorem the curvature tensor of an almost cosymplectic  $(\ka,\mu,\nu)$-space of dimension greater than or equal to $5$ and $\ka<0$.

\begin{corollary}\label{teo-est-c2}
If $M$ is an almost cosymplectic $(\ka,\mu,\nu)$-space of dimension greater than or equal to $5$ and $\ka<0$, then its Riemannian curvature tensor can be written as
\[R=-\ka R_3-R_{5,2}-\mu R_6-\nu R_8.\]
\end{corollary}
\begin{proof}
If we decompose every vector field on $M$ as  $X=\w{X}+\eta(X)\xi$, where $\w{X}$ is a  vector field orthogonal to $\xi$, we obtain:
\begin{equation*}
R(X,Y)Z=R(\w{X}, \w{Y})\w{Z}+\eta(X)R(\xi,Y)Z+\eta(Y)R(X,\xi)Z.
\end{equation*}

Using formula \eqref{eq-R-xi-x-y} and the definition of the tensors $R_1, \ldots,R_8$, it follows from a direct computation that
\[\eta(X)R(\xi,Y)Z+\eta(Y)R(X,\xi)Z=-\ka R_3 (X,Y)Z-\mu R_6(X,Y)Z+\nu R_8(X,Y)Z,\]
which substituted in the previous equation gives
\begin{equation}\label{descomposicion-c2}
R(X,Y)Z=R(\w{X}, \w{Y})\w{Z}-\ka R_3 (X,Y)Z-\mu R_6(X,Y)Z+\nu R_8(X,Y)Z.
\end{equation}

We do not know, in general,  $R(\w{X}, \w{Y})\w{Z}$ on a  $(\ka,\mu,\nu)$-space, but if we use the $D$-homothetic deformations \eqref{transf-cos} with  $\alpha=1$ and $\beta=\sqrt{-\ka}$, we obtain that the deformed manifold is a $(-1,\overline{\mu})$-space, with  $\overline{\mu}=\mu / \sqrt{-\ka}$. This is thanks to Proposition \ref{prop-dacko2}, which can be applied because the functions  $\ka, \mu, \nu$ only vary in the direction of $\xi$ (Theorem \ref{teo-xi}).

Given a vector field $\w{X}$,  orthogonal to $\xi$ with respect to $g$, then it is also orthogonal to $\overline{\xi}$ with respect to $\overline{g}$. Therefore, the next formula follows from Theorem \ref{teo-est-c} and the fact that $h \xi=0$:
\begin{equation*}
\overline{R}(\w{X}, \w{Y})\w{Z}=\overline{R}_3(\w{X}, \w{Y})\w{Z}-\overline{R}_{5,2}(\w{X}, \w{Y})\w{Z}-\overline{\mu} \overline{R}_6(\w{X}, \w{Y})\w{Z}=-\overline{R}_{5,2}(X,Y)Z,
\end{equation*}
for every vector fields $X,Y,Z$ on $M$. Moreover, if we use  \eqref{transf-cos} and \eqref{h-deformed}, we obtain that $\overline{R}_{5,2}(X,Y)Z=-1/\ka R_{5,2}(X,Y)Z$ for every  $X,Y,Z$, so
\begin{equation} \label{eq-aux}
\overline{R}(\w{X}, \w{Y})\w{Z}=\fra{1}{\ka} R_{5,2}(X,Y)Z.
\end{equation}
It is now enough to see the relation between  $R(\w{X}, \w{Y})\w{Z}$ and $\overline{R}(\w{X}, \w{Y})\w{Z}$.

If we substitute $\alpha=1$ and $\beta=\sqrt{-\ka}$ in the formula \eqref{nabla-deformed-cos} and use \eqref{xi-ka}, we obtain that
\[\overline{\nabla}_X Y=\nabla_X Y-\fra{\ka+1}{\ka} g(\phi h X,Y) \xi +\nu \eta(X) \eta(Y) \xi.\]

By the definition of the Riemannian curvature tensor
\[\overline{R} (\w{X},\w{Y})\w{Z}={\overline{\nabla}}_{\w{X}} \overline{\nabla}_{\w{Y}} \w{Z} -\overline{\nabla}_{\w{Y}} \overline{\nabla}_{\w{X}} \w{Z} -\overline{\nabla}_{[\w{X},\w{Y}]} \w{Z}\]  and the fact that  $\w{X}(\ka)=\w{Y}(\ka)=0$ (Theorem \ref{teo-xi}),  after some computations we get
\begin{equation}\label{eq-aux3}
\begin{aligned}
\overline{R}(\w{X}, \w{Y})\w{Z}&= R(\w{X}, \w{Y})\w{Z}+\fra{\ka+1}{\ka} (-g(\phi h \w{Y}, \w{Z})\overline{\nabla}_{\w{X}} \xi+g(\phi h \w{X},\w{Z}) \overline{\nabla}_{\w{Y}} \xi\\
&+(g(\phi h \w{Y},\nabla_{\w{X}} \w{Z})-g(\phi h \w{X},\nabla_{\w{Y}} \w{Z})
+\w{Y}(g(\phi h \w{X},\w{Z}))-\w{X}(g(\phi h \w{Y},\w{Z}))\\
&+g(\phi h [\w{X},\w{Y}],\w{Z}) )\xi ).
\end{aligned}
\end{equation}

On the other hand, $\overline{\nabla}_{\w{X}} \xi=-\sqrt{-\ka} \overline{\nabla}_{\w{X}} \overline{\xi}=\sqrt{-\ka} \overline{\phi} \overline{h} \w{X}=-\phi h \w{X}$, so
\[-g(\phi h \w{Y},\w{Z})\overline{\nabla}_{\w{X}} \xi+g(\phi h \w{X},\w{Z}) \overline{\nabla}_{\w{Y}} \xi=g(\phi h \w{Y},\w{Z})\phi h \w{X}-g(\phi h \w{X},\w{Z}) \phi h \w{Y}=R_{5,2}(\w{X},\w{Y})\w{Z}.\]

By the properties of the Levi-Civita connection and the equation \eqref{nabla-phi-h}, we have
\begin{gather*}
g(\phi h \w{Y},\nabla_{\w{X}} \w{Z})-g(\phi h \w{X},\nabla_{\w{Y}} \w{Z})+\w{Y}(g(\phi h \w{X},\w{Z}))-\w{X}(g(\phi h \w{Y},\w{Z}))\\
+g(\phi h [\w{X},\w{Y}],\w{Z}) =g((\nabla_{\w{Y}} \phi h) \w{X}-(\nabla_{\w{X}} \phi h )\w{Y},\w{Z})=0.
\end{gather*}

Using the last two formulas in \eqref{eq-aux3}, we get
\[\overline{R}(\w{X}, \w{Y})\w{Z}=R(\w{X}, \w{Y})\w{Z}+\fra{\ka+1}{\ka} R_{5,2} (X,Y)Z, \]
with substituting in \eqref{eq-aux} gives
\begin{equation}\label{eq-aux2}
R(\w{X}, \w{Y})\w{Z}=- R_{5,2} (X,Y)Z,
\end{equation} for every vector fields $X,Y,Z$ on $M$.

Substituting \eqref{eq-aux2} in \eqref{descomposicion-c2}, we conclude that
\[R(X,Y)Z=-\ka R_3(X,Y)Z-\mu R_6(X,Y)Z+\nu R_8(X,Y)Z-R_{5,2}(X,Y)Z,\] for every $X,Y,Z$ vector fields on $M$.
\end{proof}

Corollary \ref{teo-est-c2} suggests that it would be useful to introduce the concept of \emph{g.$(\ka,\mu,\nu)$-s.f. with divided $R_5$} in the same way that it was done with \emph{g.$(\ka,\mu)$-s.f.'s with divided $R_5$} in  \cite{nuestro2}.
\begin{definition}
A  \emph{g.$(\ka,\mu,\nu)$-s.f. with divided $R_5$} $M(f_1,\ldots,f_{5,1},f_{5,2},\ldots,f_8)$ is an almost contact metric manifold whose curvature tensor can be written as
\[R=f_1R_1+f_2R_2+f_3R_3+f_4R_4+f_{5,1}R_{5,1}+f_{5,2}R_{5,2}+f_6R_6+f_7R_7+f_8R_8,\]
where $f_1, \ldots, f_8$ are functions on $M$, $R_1, \ldots, R_6$ the tensors   that appear in \eqref{R1-R6}, $R_{5,1}$ and $R_{5,2}$  the ones defined in \eqref{def-R51} and \eqref{def-R52}, and finally $R_7$ and $R_8$ the ones in  \eqref{def-R7} and \eqref{def-R8}.
\end{definition}

As $R_5=R_{5,1}-R_{5,2}$, it is obvious that every \emph{g.$(\ka,\mu,\nu)$-s.f.} $M(f_1,\ldots,f_8)$ is a  \emph{g.$(\ka,\mu,\nu)$-s.f. with divided  $R_5$} $M(f_1,\ldots,f_{5,1},f_{5,2},\ldots,f_8)$ with $f_{5,1}=f_5$ and $f_{5,2}=-f_5$.

By virtue of Corollary \ref{teo-est-c2}, every almost cosymplectic $(\ka,\mu,\nu)$-space of dimension greater than or equal to $5$ with $\ka<0$ is a  \emph{g.$(\ka,\mu,\nu)$-s.f. with divided $R_5$} $M(f_1,\ldots,f_{5,1},f_{5,2},\ldots,f_8)$ with functions:
\[f_1=f_2=f_4=f_{5,1}=f_7=0, \ f_3=-\ka, \ f_{5,2}=-1, \ f_6=-\mu, \ f_8=-\nu.\]

We will later see that these manifolds cannot be \emph{g.$(\ka,\mu,\nu)$-s.f.'s} $M(f_1,\ldots,f_8)$, i.e. their curvature tensor cannot be written without dividing $R_5$. This will justify the definition of the \emph{g.$(\ka,\mu,\nu)$-s.f.'s with divided $R_5$}. In fact, we will prove a more general result: the non-existence of almost cosymplectic \emph{g.$(\ka,\mu,\nu)$-s.f.'s} $M(f_1,\ldots,f_8)$ of dimension greater than or equal to $5$ and $\ka=f_1-f_3<0$.

In order to do that, we will first present a couple of results. The proof of the first one is analogous to the one of Proposition 4.1 of \cite{nuestro-murathan}, true for contact metric \emph{g.$(\ka,\mu,\nu)$-s.f.'s}. This is possible because $h$ is symmetric and anticommutes with $\phi$ on all three structures:

\begin{proposition}\label{prop-basica-dividido}
If $M(f_1,\ldots,f_{5,1},f_{5,2},\ldots,f_8)$ is a  \emph{g.$(\ka,\mu,\nu)$-s.f. with divided $R_5$} with almost cosymplectic or almost Kenmotsu structures, then it is a $(\kappa,\mu,\nu)$-space with $\kappa=f_1-f_3$, $\mu=f_4-f_6$ and $\nu=f_7-f_8$.
\end{proposition}

We will now present a result that is true for manifolds of dimension greater than or equal to $5$ but not of dimension $3$, which will be studied at the end of the section.

\begin{theorem}\label{unicidad-dividido}
If $M(f_1,\ldots,f_{5,1},f_{5,2},\ldots,f_8)$ is an almost cosymplectic \emph{g.$(\ka,\mu,\nu)$-s.f. with divided $R_5$} of dimension $2n+1 \geq5$ with $\ka<0$, then the writing of its curvature tensor is unique.
\end{theorem}
\begin{proof}
Let us suppose that we can write the Riemann curvature tensor of $M$ as
\[R=f_1R_1+f_2R_2+f_3R_3+f_4R_4+f_{5,1}R_{5,1}+f_{5,2}R_{5,2}+f_6R_6+f_7R_7+f_8R_8\]
and
\[R=f_1^* R_1+f_2^* R_2+f_3^* R_3+f_4^* R_4+f_{5,1}^* R_{5,1}+f_{5,2}^* R_{5,2}+f_6^* R_6+f_7^* R_7+f_8^*R_8,\]
for certain functions  $f_i, f_i^*, \ i=1,\ldots,8$.

Then
\begin{equation}\label{ecuacion=0}
\begin{aligned}
(f_1-f_1^*)R_1+(f_2-f_2^*)R_2+(f_3-f_3^*)R_3+(f_4-f_4^*)R_4+(f_{5,1}-f_{5,1}^*) R_{5,1} \\
+(f_{5,2}-f_{5,2}^*) R_{5,2}+(f_6-f_6^*) R_6+(f_7-f_7^*) R_7+(f_8-f_8^*) R_8=0.
\end{aligned}
\end{equation}

On the other hand, we know by Proposition \ref{prop-basica-dividido} that $M$ is a $(\ka,\mu,\nu)$-space with
\begin{align}
\ka&=f_1-f_3=f_1^*-f_3^*,      \nonumber\\
\mu&=f_4-f_6=f_4^*-f_6^*,      \label{dem-kamunu}\\
\nu&=f_7-f_8=f_7^*-f_8^*.      \nonumber
\end{align}
Applying  Proposition \ref{prop-casi-c} we obtain that
\begin{equation}\label{dem-unicidad-dividido}
TM=D(\la) \oplus D(-\la) \oplus <\xi>,
\end{equation}
where $\pm \la$ are the eigenvalues of the operator $h$, $\la>0$ and  $dim D(\la)=dim D(-\la)=n \geq 2$.

If we take in the equation \eqref{ecuacion=0} vector fields $X,Y \in D(\la)$ unit and mutually orthogonal and  $Z=\phi X$, then
\[f_2-f_2^*+\la^2(f_{5,2}-f_{5,2}^*)=0, \quad f_7-f_7^*=0.\]

If we choose $X, Z \in D(-\la)$ unit and mutually orthogonal and $Y=\phi Z$, the space satisfies that
\[-(f_2-f_2^*)+\la^2(f_{5,2}-f_{5,2}^*)=0, \quad f_7-f_7^*=0.\]

Taking $X, Z \in D(\la)$ unit and mutually orthogonal and $Y=\phi X$, we get that \[f_2-f_2^*=0.\]

If $X=Z \in D(\la)$ is unit and $Y=\phi X$, then
\[-3(f_2-f_2^*)+\la^2(f_{5,1}-f_{5,1}^*)+\la^2(f_{5,2}-f_{5,2}^*)=0.\]

If we choose $X, Y \in D(\la)$ unit and mutually orthogonal and  $Z= X$, we deduce that
\[f_1-f_1^*+2\la(f_4-f_4^*)+\la^2 (f_{5,1}-f_{5,1}^*)=0,   \quad  f_7-f_7^*=0.\]

Analogously, taking $X, Y \in D(-\la)$ unit and mutually orthogonal and $Z= Y$, then
\[f_1-f_1^*-2\la(f_4-f_4^*)+\la^2 (f_{5,1}-f_{5,1}^*)=0,   \quad  f_7-f_7^*=0.\]

Gathering all the  equations, we obtain the next system:
\[\left.
\begin{array}{rcc}
f_2-f_2^*+\la^2(f_{5,2}-f_{5,2}^*)&=&0\\
-(f_2-f_2^*)+\la^2(f_{5,2}-f_{5,2}^*)&=&0\\
f_2-f_2^*&=&0\\
-3(f_2-f_2^*)+\la^2(f_{5,1}-f_{5,1}^*)+\la^2(f_{5,2}-f_{5,2}^*)&=&0\\
f_1-f_1^*+2\la(f_4-f_4^*)+\la^2 (f_{5,1}-f_{5,1}^*)&=&0\\
f_1-f_1^*-2\la(f_4-f_4^*)+\la^2 (f_{5,1}-f_{5,1}^*)&=&0\\
f_7-f_7^*&=&0
\end{array} \right\}\]
which can be solved using that $\la>0$. Its solution would be:
\[f_1-f_1^*=f_2-f_2^*=f_4-f_4^*=f_{5,1}-f_{5,1}^*=f_{5,2}-f_{5,2}^*=f_7-f_7^*=0.\]

By equations \eqref{dem-kamunu}, we also have that
\[f_3-f_3^*=f_6-f_6^*=f_8-f_8^*=0,\]
and we conclude that $f_i=f_i^*$ for every $i=1,\ldots,8$. Therefore, the writing of the curvature tensor is unique.
\end{proof}

It is worth remarking that the previous theorem is also true if the structure is contact metric and $\ka<1$ or if the structure is almost Kenmotsu and $\ka<-1$:

\begin{theorem}\label{unicidad-dividido2}
If $M(f_1,\ldots,f_{5,1},f_{5,2},\ldots,f_8)$ is a contact metric (resp. almost Kenmotsu) \emph{g.$(\ka,\mu,\nu)$-s.f. with divided $R_5$}
of dimension $2n+1 \geq 5$ with $\ka<1$ (resp. $\ka<-1$), then the writing of its curvature tensor is unique.
\end{theorem}

%%%%%%%%%%%%%%%%%%%%%%%%%%%%%%%%5

We can now prove the result we wanted.

\begin{proposition}\label{noexist-cos}
There are no almost cosymplectic \emph{g.$(\ka,\mu,\nu)$-s.f.'s} $M(f_1,\ldots,f_8)$ of dimension greater than or equal to  $5$ with $\ka=f_1-f_3<0$.
\end{proposition}
\begin{proof}
Suppose that there exists an almost cosymplectic \emph{g.$(\ka,\mu,\nu)$-s.f.} $M(f_1,\ldots,f_8)$ of dimension greater than or equal to $5$ with $\ka=f_1-f_3<0$. Then  $M$ is also a  \emph{g.$(\ka,\mu,\nu)$-s.f. with divided $R_5$} $M(f_1,\ldots,f_{5,1},f_{5,2},\ldots,f_8)$ with $f_{5,1}=f_5$ and $f_{5,2}=-f_5$, so
\[R=f_1R_1+f_2R_2+f_3R_3+f_4R_4+f_5R_{5,1}-f_5 R_{5,2}+f_6R_6+f_7R_7+f_8R_8.\]

On the other hand, as the manifold is almost cosymplectic, it is also a $(f_1-f_3,f_4-f_6,f_7-f_8)$-space with $f_1-f_3<0$,  $f_4-f_6$ and $f_7-f_8$ varying only in the direction of $\xi$  (Proposition \ref{prop-basica-dividido} and Theorem \ref{teo-xi}). Hence we can apply  Corollary \ref{teo-est-c2} and obtain that the curvature tensor can be written as
\[R=-(f_1-f_3)R_3-R_{5,2}-(f_4-f_6)R_6-(f_7-f_8) R_8.\]

By Theorem \ref{unicidad-dividido}, the functions of both writings of the curvature tensor must coincide, so we would have in particular that $f_5=0$ and $f_5=1$, which is absurd. The non-existence is thus proved.
\end{proof}

\begin{example}\label{eje-cos-nu}
We already showed in  Example  \ref{eje-cos} that  P. Dacko and Z. Olszak gave in  \cite{dacko2005-b} models of almost cosymplectic $(-1,\mu,0)$-spaces. If we $D$-homothetically deform them, we obtain almost cosymplectic $(\overline{\ka},\overline{\mu},\overline{\nu})$-spaces with \[\overline{\ka}=-\fra{1}{\beta^2}<0, \ \overline{\mu}=\fra{\mu}{\beta}, \ \overline{\nu}=-\fra{\xi(\beta)}{\beta^2}.\]

By the previous remark, if these deformed spaces have dimension greater than or equal to $5$, then they are
\emph{g.$(\ka,\mu,\nu)$-s.f.'s with divided $R_5$} $M(f_1,\ldots,f_{5,1},f_{5,2},\ldots,f_8)$ with functions:
\[f_1=f_2=0, \ f_3=\fra{1}{\beta^2},\ f_4=f_{5,1}=0, \ f_{5,2}=-1, \ f_6=-\fra{\mu}{\beta}, \ f_7=0, \ f_8=\fra{\xi(\beta)}{\beta^2}.\]
It is worth noting that, in this case, $f_{5,1}\neq -f_{5,2}$ and that $f_8$ is, in general, a non-constant function, hence these spaces are not \emph{g.$(\ka,\mu,\nu)$-s.f.'s} $M(f_1,\ldots,f_8)$ or  \emph{g.$(\ka,\mu)$-s.f.'s with divided $R_5$} $M(f_1,\ldots,f_{5,1},f_{5,2},f_6)$. This further justifies the definition and study of the \emph{g.$(\ka,\mu,\nu)$-s.f.'s with divided $R_5$} $M(f_1,\ldots,f_{5,1},f_{5,2},\ldots,f_8)$.
\end{example}

We can also use Theorem \ref{unicidad-dividido} to determine some of the functions of  an almost cosymplectic \emph{g.$(\ka,\mu,\nu)$-s.f. with divided $R_5$} $M(f_1,\ldots,f_{5,1},f_{5,2},\ldots,f_8)$ of dimension greater than or equal to $5$ with $f_1-f_3<0$.

\begin{theorem}\label{teodim5dividido-cos}
Let $M(f_1,\ldots,f_{5,1},f_{5,2},\ldots,f_8)$ be an almost cosymplectic \emph{g.$(\ka,\mu)$-s.f. with divided $R_5$}. If its dimension is greater than or equal to $5$ and $f_1-f_3<0$, then $M$ satisfies
\[f_1=f_2=f_4=f_{5,1}=f_7=0, \  f_{5,2}=-1, \ f_3>0,\]
and $f_3,f_6,f_8$ are functions that only vary in the direction of $\xi$, i.e.  $M$ is a $(-f_3,-f_6,-f_8)$-space with $f_3>0$.
\end{theorem}
\begin{proof}
By Proposition \ref{prop-basica-dividido}, we know that $M$ is a $(\ka,\mu,\nu)$-space with $\ka=f_1-f_3<0$, $\mu=f_4-f_6$ and $\nu=f_7-f_8$. Therefore, we can apply Corollary \ref{teo-est-c2}, which tells us that the Riemann curvature tensor can be written as
\[R=-\ka R_3-R_{5,2}-\mu R_6-\nu R_8=-(f_1-f_3) R_3-R_{5,2}-(f_4-f_6) R_6-(f_7-f_8) R_8.\]
By the definition of \emph{g.$(\ka,\mu,\nu)$-s.f. with divided  $R_5$} and the uniqueness of the writing of the curvature tensor (Theorem \ref{unicidad-dividido}), we obtain that
\[f_1=f_2=f_4=f_{5,1}=f_7=0, \  f_{5,2}=-1.\]
Therefore, $\ka=-f_3$, $\mu=-f_6$ and $\nu=-f_8$. By hypothesis, $f_1-f_3=-f_3<0$, so $f_3>0$. The rest of the result is a consequence of Theorem \ref{teo-xi}.
\end{proof}

%%%%%%%%%%%%%%%%%%%%%%%%%%%%%%%%%%%%%%%%%%%%%%%%%%%%%%%%%%%%%%%%%%%%%%%%%%%%%%%%%%
%%Estudiaremos el caso casi-Kenmotsu
%%%%%%%%%%%%%%%%%%%%%%%%%%%%%%%%%%%%%%%%%%%%%%%%%%%%%%%%%%%%%%%%%%%%%%%%%%%%%%%%%%

We will now analyse the  almost Kenmotsu  $(\ka,\mu,\nu)$-spaces, which have been studied by other authors in some particular cases. For instance, G. Dileo and A. M. Pastore  studied in \cite{dileo09} the almost Kenmotsu $(\ka,\mu)$-spaces and  $(\ka,0,\nu)$-spaces with constant $\ka,\mu,\nu$, although they used a different notation. Among other results, they proved that, if a $(\ka,\mu)$-space is almost Kenmotsu, where  $\ka,\mu \in \mathbb{R}$, then $\ka=-1$ and $h=0$ (Theorem 4.1). If a $(\ka, 0,\nu)$-space is almost Kenmotsu, with $h \neq 0$, then $\ka < -1$ and $\nu=2$ (Proposition 4.1). They also gave examples of almost Kenmotsu $(-1-\la^2,0,2)$-spaces for all $\la \in \mathbb{R}$.

The equation \eqref{h-phi}  of  Proposition \ref{prop-basica} (with $\alpha=1$) makes $\ka$ play an important role when analysing the behaviour of the almost Kenmotsu $(\ka,\mu,\nu)$-spaces. We will now study these spaces distinguishing between $\ka=-1$ and $\ka<-1$, obtaining results that generalise those appearing in \cite{dileo09}. The proofs are similar and will therefore be omitted.

\begin{proposition}\label{prop-casi-k2}
Let $M^{2n+1}$ be an almost Kenmotsu $(\ka,\mu,\nu)$-space.

If $\ka =-1$, then $h=0$ and $M^{2n+1}$ is locally the warped product $M' \times_{f^2} N^{2n}$, where $N^{2n}$ is an almost K\"{a}hler manifold, $M'$ an open interval of coordinate $t$ and $f^2=ce^{2t}$ for some positive constant $c$. Moreover, it is a Kenmotsu manifold if the dimension is $3$ ($n=1$).

If $\ka<-1 \  (h \neq 0)$, then $M$ is not Kenmotsu and the eigenvalues of $h$ (equal to those of $\phi h$) are $0$ (with multiplicity $1$) and $\pm \la=\pm  \sqrt{-1-\ka} \neq 0$ (each one with multiplicity $n$). Moreover,  $\mu=-2 g(\nabla_\xi X,\phi X)$ holds for every $X$ unit eigenvector  of $h$ associated to $\pm \la$.
\end{proposition}

As it happened in the almost cosymplectic case, we know from Theorem \ref{teo-xi} that, on an almost Kenmotsu $(\ka,\mu,\nu)$-space of dimension greater than or equal to $5$, the functions $\ka$, $\mu$, and $\nu$ only vary in the direction of $\xi$.

We will now study the writing of the curvature tensor of an almost Kenmotsu $(\ka,\mu,\nu)$-space with  $\ka<-1$ (if $\ka=-1$, we know its local structure by Proposition \ref{prop-casi-k2}). We will first present some results that generalise the ones  appearing in \cite{dileo09} for $(\ka,0,\nu)$-spaces, where $\ka$ and $\nu$ are constant. The proofs are analogous and will be omitted.

\begin{proposition}
Let $M$ be an almost Kenmotsu $(\ka,\mu,\nu)$-space with $\ka<-1$. Then
\begin{equation*}
\begin{aligned}
R(X,Y)\phi Z&-\phi R(X,Y)Z=\\
&=g(\phi^2 X+\phi h X,Z)(\phi Y+h Y)-g(\phi^2 Y+\phi h Y,Z)(\phi X+h X)\\
&+g(\phi X+h X,Z)(\phi^2 Y+\phi h Y)-g(\phi Y+h Y,Z)(\phi^2 X+\phi h X)\\
&+\ka(\eta(Y)(g(\phi X,Z)\xi-\eta(Z) \phi X)-\eta(X)(g(\phi Y,Z)\xi-\eta(Z)\phi Y))\\
&-\nu R_6(X,Y)Z+\mu R_8(X,Y)Z,
\end{aligned}
\end{equation*}
for every $X,Y,Z$ vector fields on $M$.

In particular, if $X,Y,Z$ are orthogonal to $\xi$, then
\begin{equation}\label{prop1-ken}
\begin{aligned}
R(X,Y)\phi Z&-\phi R(X,Y)Z=\\
&=g(-X+\phi h X,Z)(\phi Y+h Y)-g(-Y+\phi h Y,Z)(\phi X+h X)\\
&-g(\phi Y+h Y,Z)(-X+\phi h X)+g(\phi X+h X,Z)(- Y+\phi h Y).
\end{aligned}
\end{equation}
\end{proposition}

\begin{lemma}
Let $M^{2n+1}$ be an almost Kenmotsu $(\ka,\mu,\nu)$-space with $\ka<-1$ a function that only varies in the direction of $\xi$. Then
\begin{equation*}
\begin{aligned}
(\nabla_X \phi h)Y&=g((\ka+1)\phi^2 X-\phi h X,Y)\xi\\
&+\eta(Y)((\ka+1)\phi^2 X-\phi h X)+\eta(X)(\mu h Y+(\nu-2)\phi h Y),
\end{aligned}
\end{equation*}
for every  $X,Y,Z$ vector fields on $M$.
\end{lemma}

\begin{proposition}
Let $M$ be an almost Kenmotsu $(\ka,\mu,\nu)$-space with $\ka<-1$ a function that only varies in the direction of $\xi$. Then
\begin{gather}
R(X,Y)\phi h Z -\phi h  R(X,Y)Z=  \label{prop2-ken}\\
=(\ka+2)(g(Y,Z)\phi h X-g(X,Z)\phi h Y+g(\phi h X,Z)Y-g(\phi h Y,Z)X), \nonumber
\end{gather} for every $X,Y,Z$ vector fields orthogonal to $\xi$.
\end{proposition}

Using the equations \eqref{prop1-ken} and \eqref{prop2-ken}, we can prove the following theorem analogously to Proposition 4.2 of \cite{dileo09}.

\begin{theorem}\label{teo-ken}
If $M$ is an almost Kenmotsu $(\ka,\mu,\nu)$-space with $\ka<-1$ a function that only varies in the direction of $\xi$, then
\begin{equation*}
\begin{aligned}
&R(X_{\la},Y_{\la})Z_{-\la}&=& \quad 0,\\
&R(X_{-\la},Y_{-\la})Z_{\la}&=&\quad 0,\\
&R(X_{\la},Y_{-\la})Z_{-\la}&=&  \quad -(\ka+2)g(Y_{-\la},Z_{-\la})X_{\la}, \\
&R(X_{\la},Y_{-\la})Z_{\la}&=& \quad (\ka+2)g(X_\la,Z_\la)Y_{-\la} ,\\
&R(X_\la,Y_\la)Z_\la&=&\quad (\ka+2 \la)(g(Y_\la,Z_\la)X_\la-g(X_\la,Z_\la)Y_\la),\\
&R(X_{-\la},Y_{-\la})Z_{-\la}&=& \quad (\ka - 2 \la) (g(Y_{-\la},Z_{-\la})X_{-\la}-g(X_{-\la},Z_{-\la})Y_{-\la}),
\end{aligned}
\end{equation*}
where $X_{\pm \la},Y_{ \pm \la},Z_{\pm \la}$  are eigenvectors of $\phi h$ associated to eigenvalues $\pm \la=\pm \sqrt{-1-\ka}$.
\end{theorem}

By the previous theorem and formula \eqref{eq-R-xi-x-y}, we can give explicitly the expression of the curvature tensor of an almost Kenmotsu $(\ka,\mu,\nu)$-space with $\ka<-1$:

\begin{theorem}\label{teo-est-k}
If $M$ is an almost Kenmotsu $(\ka,\mu,\nu)$-space with $\ka < -1$ a function that only varies in the direction of $\xi$, then its Riemann curvature tensor can be written as
\[R=-R_1-(\ka+1) R_3-R_{5,2}-\mu R_6+R_7-(\nu-1)R_8.\]
Therefore, $M$ is a \emph{g.$(\ka,\mu,\nu)$-s.f. with divided $R_5$} $M(f_1,\ldots,f_{5,1},f_{5,2},\ldots,f_8)$ with functions
\begin{align*}
f_1&=-1,    & f_2&=0,      & f_3&=-(\ka+1),   &  f_4&=0,\\
f_{5,1}&=0, & f_{5,2}&=-1, & f_6&=-\mu,       &  f_7&=1,  & f_8=-(\nu-1).
\end{align*}
\end{theorem}
\begin{proof}
As $\ka < -1$, we know by Proposition \ref{prop-casi-k2} that $TM=D'(\lambda) \oplus D'(-\lambda) \oplus <\xi> $, where $\lambda=\sqrt{-1-\ka}$ . Therefore, given a differentiable vector field $X$ on  $M$, we can write $X=X_\lambda+X_{-\lambda}+\eta(X)\xi$, where $X_{\pm \lambda}$ is an eigenvector of $\phi h$ associated to the eigenvalue $\pm \lambda$, i.e. $X_{\pm \lambda} \in D'(\pm \la)$. Hence, by the properties of the Riemann curvature tensor $R$ we obtain that
\begin{equation}\label{descomposicion-k}
\begin{aligned}
R(X,Y)Z=&R(X_\lambda+X_{-\lambda},  Y_\lambda+Y_{-\lambda})(Z_\lambda+Z_{-\lambda})\\
&+\eta(X)R(\xi,Y)Z+\eta(Y)R(X,\xi)Z.
\end{aligned}
\end{equation}

By formula \eqref{eq-R-xi-x-y} and the definition of the tensors $R_1, \ldots,R_8$, we can prove that
\begin{equation}\label{intermedia4}
\begin{split}
\eta(X)R(\xi,Y)Z&+\eta(Y)R(X,\xi)Z=\\
&-\ka R_3 (X,Y)Z-\mu R_6(X,Y)Z-\nu R_8(X,Y)Z.
\end{split}
\end{equation}

By virtue of Theorem \ref{teo-ken}, we obtain
\begin{equation*}
\begin{aligned}
R( &X_\lambda +X_{-\lambda}, Y_\lambda+Y_{-\lambda})(Z_\lambda+Z_{-\lambda})=
(\ka+2\la)(g(Y_\la,Z_\la)X_\la -g(X_\la,Z_\la)Y_\la)\\
+&(\ka-2\la)(g(Y_{-\la},Z_{-\la})X_{-\la}-g(X_{-\la},Z_{-\la})Y_{-\la})\\
+&(\ka+2)(g(X_\la,Z_\la)Y_{-\la}-g(Y_{-\la},Z_{-\la})X_\la -g(Y_\la,Z_\la)X_{-\la}+g(X_{-\la},Z_{-\la})Y_\la).
\end{aligned}
\end{equation*}
From the decomposition $X=X_\lambda+X_{-\lambda}+\eta(X)\xi$, we deduce that
\[ X_\la=\fra{1}{2} \left( X-\eta(X)\xi+\fra{1}{\la}\phi h X \right), \quad
X_{-\la}=\fra{1}{2} \left(X-\eta(X)\xi-\fra{1}{\la} \phi hX \right), \]
so by a direct computation:
\begin{equation*}
\begin{aligned}
g(Y_\la,Z_\la)X_\la& -g(X_\la,Z_\la)Y_\la=\\
&=\fra{1}{4}\left((R_1+R_3)+\fra{1}{\la^2} R_{5,2}+\fra{1}{\la} (R_7+R_8)\right) (X,Y)Z,\\
g(Y_{-\la},Z_{-\la})X_{-\la}&-g(X_{-\la},Z_{-\la})Y_{-\la}=\\
&=\fra{1}{4}\left((R_1+R_3)+\fra{1}{\la^2} R_{5,2} -\fra{1}{\la} (R_7+R_8) \right)(X,Y)Z,\\
g(X_\la,Z_\la)Y_{-\la}&-g(Y_{-\la},Z_{-\la})X_\la  -g(Y_\la,Z_\la)X_{-\la}+g(X_{-\la},Z_{-\la})Y_\la=\\
&=\fra{1}{2} \left(-(R_1+R_3)+\fra{1}{\la^2}R_{5,2} \right)(X,Y)Z.
\end{aligned}
\end{equation*}

Therefore, it follows from  $\la^2=-(\ka+1)$  that
\begin{equation}\label{intermedia5}
R(X_\lambda  +X_{-\lambda}, Y_\lambda  +Y_{-\lambda})(Z_\lambda+Z_{-\lambda})=(-R_1-R_3-R_{5,2}+R_7+R_8)(X,Y)Z.
\end{equation}
Substituting \eqref{intermedia4} and \eqref{intermedia5} in \eqref{descomposicion-k} gives us
\begin{equation*}
R=-R_1-(\ka+1) R_3-R_{5,2}-\mu R_6+R_7-(\nu-1)R_8,
\end{equation*}
the formula we were looking for.
\end{proof}

\begin{remark}
By Theorem \ref{teo-xi}, we can omit in the previous theorem the hypothesis ``$\ka$ is a function that only varies in the direction of $\xi$'' if the dimension is greater than or equal to $5$. In dimension $3$ it is possible to simplify the writing of the curvature tensor, as we will see later in Corollary \ref{cor-R-otras}.
\end{remark}

\begin{example}\label{eje-ken}
By the previous theorem, the examples that G. Dileo and A. M. Pastore gave in \cite{dileo09} of almost Kenmotsu $(-1-\la^2,0,2)$-spaces, with $\la$ a positive real number, are in particular \emph{g.$(\ka,\mu,\nu)$-s.f.'s with divided $R_5$} $M(f_1,\ldots,f_{5,1},f_{5,2},\ldots,f_8)$ with functions
\begin{align*}
f_1&=-1,    & f_2&=0,      & f_3&=\la^2>0,   &  f_4&=0,\\
f_{5,1}&=0, & f_{5,2}&=-1, & f_6&=0,         &  f_7&=1, & f_8=-1.
\end{align*}
Moreover, as $f_7, f_8 \neq0$ and the writing of the curvature tensor is unique if these examples are of dimension greater than or equal to $5$ (Theorem \ref{unicidad-dividido2}), we know that they cannot be \emph{g.$(\ka,\mu)$-s.f.'s with divided $R_5$} $M(f_1,\ldots,f_{5,1},f_{5,2},f_6)$.
\end{example}

The almost Kenmotsu $(\ka,\mu,\nu)$-spaces of dimension greater than or equal to $5$ with  $\ka<-1$ cannot be \emph{g.$(\ka,\mu,\nu)$-s.f.'s}, i.e. their curvature tensors cannot be written without dividing $R_5$ in $R_{5,1}$ and $R_{5,2}$, as a consequence of Theorems \ref{unicidad-dividido2} and \ref{teo-est-k}. Moreover, we can prove a more general result when the dimension is greater than or equal to $5$.

\begin{proposition}\label{noexist-ken}
There are no almost Kenmotsu \emph{g.$(\ka,\mu,\nu)$-s.f.'s} $M(f_1,\ldots,f_8)$ of dimension greater than or equal to $5$ and $\ka=f_1-f_3 <-1$.
\end{proposition}
\begin{proof}
We will prove the result by contradiction. Let us suppose that $M(f_1,\ldots,f_8)$ is an almost Kenmotsu \emph{g.$(\ka,\mu,\nu)$-s.f.} of dimension greater than or equal to $5$ with $\ka<-1$. Then $M$ is also a \emph{g.$(\ka,\mu,\nu)$-s.f. with divided $R_5$} with $f_{5,1}=f_5$ and $f_{5,2}=-f_5$, so
\[R=f_1R_1+f_2 R_2+f_3R_3+f_4 R_4  +f_5 R_{5,1}-f_5R_{5,2}+f_6R_6+f_7R_7+f_8R_8.\]

On the other hand, $M$ is a $(f_1-f_3,f_4-f_6,f_7-f_8)$-space with $\ka=f_1-f_3<-1$ a function that only varies in the direction of $\xi$ (Proposition \ref{prop-basica-dividido} and Theorem \ref{teo-xi}). Applying Theorem \ref{teo-est-k} we obtain that its curvature tensor can be written as \[R=-R_1-(f_1-f_3+1)R_3-R_{5,2}-(f_4-f_6)R_6+R_7-(f_7-f_8-1)R_8.\]

By Theorem \ref{unicidad-dividido2}, the writing of the curvature tensor is unique, so we have in particular that $f_5=0$ and $f_5=1$, which is absurd.
\end{proof}

We can also use Theorem  \ref{unicidad-dividido2} to determine some relations between the functions of an almost Kenmotsu \emph{g.$(\ka,\mu,\nu)$-s.f. with divided $R_5$} $M(f_1,\ldots,f_{5,1},f_{5,2},\ldots,f_8)$ of dimension greater than or equal to $5$ with $f_1-f_3<-1$, as we did in Theorem \ref{teodim5dividido-cos} for the almost cosymplectic structure.

\begin{theorem}\label{teodim5dividido-ken}
Let $M(f_1,\ldots,f_{5,1},f_{5,2},\ldots,f_8)$  be an almost Kenmotsu \emph{g.$(\ka,\mu)$-s.f. with divided $R_5$}. If $M$ is of dimension greater than or equal to $5$ and satisfies  $f_1-f_3<-1$, then $M$ verifies
\[f_1=-1, \ f_2=f_4=f_{5,1}=0, \  f_{5,2}=-1, \ f_7=1, \ f_3>0,\]
and $f_3,f_6,f_8$ are functions that only vary in the direction of $\xi$, i.e. $M$ is a $(-1-f_3,-f_6,1-f_8)$-space with $f_3>0$.
\end{theorem}
\begin{proof}
By Theorem \ref{prop-basica-dividido}, we know that $M$ is a $(\ka,\mu,\nu)$-space with $\ka=f_1-f_3<-1$, $\mu=f_4-f_6$ and $\nu=f_7-f_8$. We can therefore apply Theorem \ref{teo-est-k}, which says that the Riemann curvature tensor can be written as
\begin{equation*}
R=-R_1-(f_1-f_3+1) R_3-R_{5,2}-(f_4-f_6) R_6+R_7-(f_7-f_8-1) R_8.
\end{equation*}
By the definition of \emph{g.$(\ka,\mu,\nu)$-s.f. with divided  $R_5$} and the uniqueness of the writing of the curvature tensor (Theorem \ref{unicidad-dividido2}), we obtain that
\[f_1=-1, \ f_2=f_4=f_{5,1}=0, \  f_{5,2}=-1, \ f_7=1.\]
Therefore, $\ka=-1-f_3$, $\mu=-f_6$ and $\nu=1-f_8$. By hypothesis, $f_1-f_3=-1-f_3<-1$, so $f_3>0$. The rest of the result is deduced from Theorem  \ref{teo-xi}.
\end{proof}

%%%%%%%%%%%%%%%%%%%%%%%%%%%%%%%%%%%%%%%%%%%%%%%%%%%%%%%%%%%%%%%%%%%%%%%%%%%%%%%%%%%%%%%%%%%%%%%%%%
%%%dimensión 3
%%%%%%%%%%%%%%%%%%%%%%%%%%%%%%%%%%%%%%%%%%%%%%%%%%%%%%%%%%%%%%%%%%%%%%%%%%%%%%%%%%%%%%%%%%%%%%%%%%

Finally, we will see what happens to an almost cosymplectic or almost Kenmotsu $(\ka,\mu,\nu)$-space of dimension $3$. Using  formula \eqref{eq-R-xi-x-y}, we can prove an analogous result to Theorem 3.1 of \cite{nuestro-murathan}.

\begin{theorem}\label{prop-R-otras}
Let  $M^3$ be an almost cosymplectic (resp. almost Kenmotsu) $(\ka,\mu,\nu)$-space with $\ka<0$ (resp. $\ka<-1$). Then its curvature tensor can be written as
\[R=\left( \fra{\tau}{2} -2 \ka \right) R_1+\left( \fra{\tau}{2} -3 \ka
\right)R_3+\mu R_4 +\nu R_7,\]
where $\tau$ is the scalar curvature of $M$ and the tensors $R_1,R_3,R_4,R_7$ are the ones that appear in \eqref{R1-R6} and \eqref{def-R7}.
\end{theorem}
\begin{proof}
Firstly, we recall a formula that is valid for every Riemannian manifold of dimension $3$:
\begin{equation}\label{formula-dim3-tau-Q}
\begin{aligned}
R(X,Y)Z&=g(Y,Z)QX-g(X,Z)QY+g(QY,Z)X-g(QX,Z)Y\\
&-\fra{\tau}{2}(g(Y,Z)X-g(X,Z)Y),
\end{aligned}
\end{equation}
where $Q$ is the Ricci operator and $\tau=trQ$.

We now take a $\phi$-basis $\{ E, \phi E, \xi \}$ such that $h E=\lambda E$, where $\lambda=\sqrt{-\ka}$ (resp. $\lambda=\sqrt{-1-\ka}$), which is possible thanks to Proposition \ref{prop-basica-cos} (resp. Proposition \ref{prop-casi-k2}). Using formula \eqref{eq-R-xi-x-y} and this basis, we can compute
\begin{align*}
Q \xi =&R(\xi, E)E+R(\xi,\phi E)\phi E+R(\xi,\xi)\xi=(\ka+\la \mu )\xi + (\ka-\la \mu) \xi=2 \ka \xi.
\end{align*}

Making $Y=Z=\xi$  in \eqref{formula-dim3-tau-Q} and using that  $Q \xi=2 \ka \xi$ and $g(Q X,Y)=g(QY,X)$, we obtain:
\begin{equation*}
R(X,\xi)\xi=\left(2\ka-\fra{\tau}{2} \right)X + \left( \fra{\tau}{2}-4\ka \right)\eta(X)\xi +Q X.
\end{equation*}
Using again formula \eqref{eq-R-xi-x-y}, it follows that
\begin{align*}
Q X=&R(X,\xi)\xi-\left(2\ka-\fra{\tau}{2} \right)X -\left( \fra{\tau}{2}-4\ka \right)\eta(X)\xi \\
=&\left(\fra{\tau}{2}-\ka \right)X + \left(3 \ka- \fra{\tau}{2} \right)\eta(X)\xi +\mu h X+\nu \phi h X.
\end{align*}
Substituting this expression of $Q X$ in \eqref{formula-dim3-tau-Q} and applying  the definitions of $R_1,R_3$, $R_4$ and $R_7$, we obtain the equation we were looking for.
\end{proof}

We will now see a result that was proved for contact metric $(\ka,\mu,\nu)$-spaces with $\ka<1$ in \cite{nuestro-murathan}:

\begin{proposition}
Let $M^3$ be an almost cosymplectic (resp. almost Kenmotsu) $(\ka,\mu,\nu)$-space with $\ka < 0$ (resp. $\ka< -1$). Then its $\phi$-sectional curvature is $F=\fra{\tau}{2} -2 \ka $.
\end{proposition}
\begin{proof}
Since the manifold is of dimension $3$, the $\phi$-sectional curvature does not depend on the choice of the $\phi$-section. Hence we can take $F=R(E,\phi E,\phi E,E)$, where $E$ is an eigenvector of $h$ of eigenvalue $\la >0$ thanks to equation \eqref{h-phi} and the fact that $\ka < 0$ (resp. equation \eqref{h-phi} and $\ka<-1$).

It follows from Proposition \ref{prop-R-otras} that
\begin{align*}
F=R(E,\phi E,\phi E ,E)=&\left( \fra{\tau}{2} -2 \ka \right) R_1 (E,\phi E,\phi E,E)+\left( \fra{\tau}{2} -3 \ka \right) R_3 (E,\phi E,\phi E,E)\\
&+\mu R_4 (E,\phi E,\phi E,E)+\nu R_7 (E,\phi E,\phi E,E).
\end{align*}

A straightforward computation using the definition of the tensors $R_1,\ldots,R_7$ gives us the formula we wanted.
\end{proof}

Therefore, Theorem \ref{prop-R-otras} can be rewritten as:

\begin{corollary}\label{prop-R-v2-otras}
Let $M^3$ an almost cosymplectic (resp. almost Kenmotsu)  $(\ka,\mu,\nu)$-space with $\ka < 0$ (resp. $\ka< -1$). Then its curvature tensor has the form
\begin{equation*}
R=F R_1+(F-\ka) R_3+\mu R_4 +\nu R_7,
\end{equation*} where $F$ is the $\phi$-sectional curvature and $R_1,R_3,R_4,R_7$ are the tensors defined in \eqref{R1-R6} and \eqref{def-R7}. In particular, $M$ is a \emph{g.$(\ka,\mu,\nu)$-s.f.} $M(F,0,F-\ka,\mu,0,0,\nu,0)$.
\end{corollary}

\begin{remark}
It is worth noting that the expression of the curvature tensor given in the previous corollary coincides with the one obtained in Corollary 3.3 of \cite{nuestro-murathan} for contact metric $(\ka,\mu,\nu)$-spaces of dimension $3$. This is because the tensor $h$ satisfies the same properties in the three structures (almost cosymplectic, almost Kenmotsu and contact metric): it is symmetric and anticommutes with $\phi$, by equations \eqref{eq-cont-h} and \eqref{nablaxi}.
\end{remark}

Under some extra hypotheses, we can write $F$ in terms of $\ka$, so the curvature tensor would be completely determined by the functions $\ka, \mu$ and $\nu$.

\begin{corollary}\label{cor-R-otras}
Let $M^3$  be an almost cosymplectic $(\ka,\mu,\nu)$-space. If $\ka<0, \mu,\nu$ only vary in the direction of $\xi$, then its curvature tensor can be written as
\[R=-\ka R_1-2 \ka R_3+\mu R_4+\nu R_7,\]
where $R_1, R_2, R_4, R_7$ are the tensors defined in \eqref{R1-R6} and \eqref{def-R7}.

Let $M^3$ be an almost Kenmotsu $(\ka,\mu,\nu)$-space. If $\ka<-1$ is a function that only varies in the direction of $\xi$, then its curvature tensor can be written as
\[R=-(\ka+2)R_1-2(\ka+1)R_3+\mu R_4+\nu R_7.\]
\end{corollary}
\begin{proof}
Since the manifold is of dimension $3$,  it follows from Proposition \ref{prop-basica-cos} that when $M$ is almost cosymplectic we can take $F=R(E,\phi E,\phi E,E)$, where $E$ is an eigenvector of $h$ associated to the eigenvalue $\la=\sqrt{-\ka} >0$.

We do not know in general the $\phi$-sectional curvature of a $(\ka,\mu,\nu)$-space, but we can use a $D$-homothetic deformation \eqref{transf-cos} with $\alpha=1$ and $\beta=\sqrt{-\ka}$ to obtain a $(-1,\overline{\mu})$-space with $\overline{\mu}=\mu / \sqrt{-\ka}$.

Reasoning analogously to the proof of Corollary \ref{teo-est-c2}, we get that, for every $X,Y,Z$ vector fields on $M$
\[\overline{R}(\w{X}, \w{Y})\w{Z}=R(\w{X}, \w{Y})\w{Z}+\fra{\ka+1}{\ka} R_{5,2} (X,Y)Z,\]
where $X=\w{X}+\eta(X) \xi$ and $\w{X}$ is orthogonal to $\xi$.

Therefore, we have in particular that
\begin{equation}\label{estonoseacaba}
F=R(E,\phi E,\phi E,E)=\overline{R}(E,\phi E,\phi E,E)-\fra{\ka+1}{\ka} R_{5,2} (E,\phi E,\phi E,E).
\end{equation}

On the other hand, since $E$ is an unit vector field with respect to $g$ and orthogonal to $\xi$, by \eqref{transf-cos} it is also unit with respect to $\overline{g}$ and orthogonal to $\overline{\xi}$, so it follows from Theorem \ref{teo-endo} that
\[\overline{R}(E,\phi E,\phi E,E)= \overline{g}(E,{\overline{\phi}}^2 E) \overline{g}(\overline{\phi}^2 E,E)=1.\]

By the definition of the tensor $R_{5,2}$  and the properties of $h$ (formulas \eqref{def-R52} and \eqref{nablaxi}):
\[R_{5,2} (E,\phi E,\phi E,E)=g(\phi h E,E)^2+g(h E,E)^2=\la^2=-\ka.\]

Substituting the last two equations in \eqref{estonoseacaba} we obtain that $F=-\ka,$ which jointly with Corollary \ref{prop-R-v2-otras} gives the result we were looking for.

If $M$ is almost Kenmotsu, then we can take $E$ an eigenvector of $\phi h$ associated to the eigenvalue $\la=\sqrt{-1-\la}\neq 0$ by virtue of Proposition \ref{prop-casi-k2}. We know by Theorem \ref{teo-ken} that $F=-(\ka+2) g(\phi E,\phi E)g(E,E)=-(\ka+2)$ and we conclude that  $R=-(\ka+2)R_1-2(\ka+1)R_3+\mu R_4+\nu R_7.$
\end{proof}

\begin{remark}
The expressions of the curvature tensor given in  Corollary \ref{teo-est-c2} and Theorem \ref{teo-est-k} coincide with that of the previous corollary in dimension $3$. This is true because the following equations hold in dimension $3$ when the structure is almost cosymplectic or almost Kenmotsu:
\begin{equation}
R_2=3(R_1+R_3), \ R_6=-R_4, \ R_8=-R_7.
\end{equation}

If the manifold is an almost cosymplectic $(\ka,\mu,\nu)$-space, it can also be proved that $R_{5,2}=\ka (R_1+R_3)$, so the curvature tensor can be written as
\begin{equation*}
R=-\ka R_3-R_{5,2}-\mu R_6-\nu R_8=-\ka R_1-2\ka R_3 +\mu R_4+\nu R_7.
\end{equation*}

If the manifold is an almost Kenmotsu $(\ka,\mu,\nu)$-space, it is also true that $R_{5,2}=(\ka+1) (R_1+R_3)$, so the curvature tensor can be written as
\begin{align*}
R=&-R_1-(\ka+1) R_3-R_{5,2}-\mu R_6+R_7-(\nu-1)R_8\\
=&-(\ka+2)R_1-2(\ka+1)R_3+\mu R_4+\nu R_7.
\end{align*}
\end{remark}

\begin{example}\label{eje-dim3-cos+ken}
By the previous corollary, the examples of almost cosymplectic  $(-1,\mu,0)$-spaces (with $\mu$ varying only in the direction of $\xi$) that appear in \cite{dacko2005-b} are in dimension  $3$ examples of  \emph{g.$(\ka,\mu)$-s.f.'s} $M^3(f_1,\ldots,f_6)$ with functions
\[f_1=1,\ f_2=0, \ f_3=2, \ f_4=\mu, \ f_5=f_6=0.\]

\

Analogously, the examples of almost Kenmotsu  $(-1-\la^2,0,2)$-spaces (with $\la$ a positive real number) that appear in \cite{dileo09} are examples of \emph{g.$(\ka,\mu,\nu)$-s.f.'s} $M^3(f_1,\ldots,f_8)$ with functions:
\[f_1=-1+\la^2, \ f_2=0, \ f_3=2 \la^2, \ f_4=f_5=f_6=0, \ f_7=2, \ f_8=0.\]
\end{example}

\

Finally, it remains to be seen what happens if a 3-dimensional manifold is almost cosymplectic and satisfies $\ka=0$ or almost Kenmotsu and satisfies $\ka=-1$. In order to study the curvature tensor, we recall the next result:

\begin{proposition}[\cite{olszak91}]\label{prop-beta}
If $M^3$ is a trans-Sasakian manifold $(0,\beta)$, then its curvature tensor can be written as
\begin{equation*}
R=\left( \fra{ \tau}{2}+2\beta^2+2\beta' \right) R_1+ \left(\fra{\tau}{2}+3\beta^2+3\beta' \right)R_3,
\end{equation*}  where $\tau$ is the scalar curvature of the manifold.
%Moreover, the Ricci operator can be written as
%\[Q=\left(\fra{ \tau}{2}+\beta^2+\beta' \right)I- \left( \fra{\tau}{2}+3\beta^2+3 \beta' \right)\eta \otimes \xi.\]
\end{proposition}

If $M^3$ is an almost cosymplectic $(\ka,\mu,\nu)$-space with $\ka=0$, then $h=0$ and $M$ is cosymplectic by Proposition \ref{prop-basica-cos}. By Proposition \ref{prop-beta} (with $\beta=0$), we have that the curvature tensor of the manifold can be written as  $R=\fra{\tau}{2}R_1+\fra{\tau}{2}R_3,$ which coincides with the result obtained in Theorem \ref{prop-R-otras}.

Analogously, if $M^3$ is an almost Kenmotsu $(\ka,\mu,\nu)$-space and $\ka=-1$, then $h=0$ and $M$ is a Kenmotsu manifold by Proposition \ref{prop-casi-k2}. By virtue of Proposition \ref{prop-beta} (with $\beta=-1$), we obtain that the curvature tensor of the manifold can be written as $R=\left(\fra{\tau}{2}+2 \right) R_1+\left(\fra{\tau}{2}+3 \right)R_3,$ which again coincides with Theorem \ref{prop-R-otras}.

\end{document}